\documentclass[letterpaper, 10 pt, conference]{ieeeconf}
\IEEEoverridecommandlockouts
\usepackage[utf8]{inputenc} 
\usepackage{dsfont}
\usepackage{cite}

\usepackage{amsfonts}
\usepackage{amsthm}
\usepackage{amsmath}
\usepackage{amssymb}
\usepackage{hyperref}       
\usepackage{url}            
\usepackage{booktabs}       
\usepackage{nicefrac}       
\usepackage{microtype}      
\usepackage{mathrsfs}
\usepackage{graphicx}
\usepackage{textcomp}
\usepackage{xcolor}
\usepackage{subcaption}
\usepackage{cleveref}

\newtheorem{theorem}{Theorem}

\newtheorem{lemma}{Lemma}

\newtheorem{remark}{Remark}
\newtheorem{assumption*}{Assumption}
\newtheorem{stdassumption*}{Standing Assumption}

\newtheorem{assumption}{Assumption}

\newtheorem{definition*}{Definition}

\DeclareMathOperator*{\argmin}{arg\,min}

\DeclareMathOperator*{\R}{\mathbb{R}}

\DeclareMathOperator*{\dom}{\mathrm{dom}}

\def\BibTeX{{\rm B\kern-.05em{\sc i\kern-.025em b}\kern-.08em
    T\kern-.1667em\lower.7ex\hbox{E}\kern-.125emX}}
    
\begin{document}

\title{\bf \LARGE 
A Unified Variational Design of Predictive Mirror Descent in Convex Games under Stochastic Feedback
}

\author{ Yunian Pan\textsuperscript{1},~\IEEEmembership{Student Member, IEEE}, Tao Li\textsuperscript{2,\textdagger},~\IEEEmembership{Member, IEEE}, and Quanyan Zhu\textsuperscript{1},~\IEEEmembership{Senior Member, IEEE}
\thanks{\textsuperscript{1}Department of Electrical and Computer Engineering, New York University, Brooklyn, NY, 11201, USA (\texttt{\{yp1170, qz494\}@nyu.edu}).  
  }
  \thanks{\textsuperscript{2}Department of Systems Engineering, City University of Hong Kong, Hong Kong SAR, 999077, China (\texttt{li.tao@cityu.edu.hk}). \textdagger Correspondence.}
}
\maketitle
\begin{abstract}
Mirror descent provides a geometric framework for learning in games, but its last-iterate behavior can fail in weakly stable regimes, where the dynamics may exhibit rotational or recurrent transients. Predictive mirror methods mitigate this issue by modifying the feedback entering the mirror update, yet standard predictive variants are typically introduced algorithmically and analyzed one at a time. This letter gives a variational route to predictive feedback by constructing a stochastic mirror differential game with an auxiliary memory state. Its stage cost couples two Fenchel terms: a strategic term evaluated at a predicted profile and a corrective term driven by realized feedback. The resulting equilibrium feedback induces two-channel predictive mirror dynamics in general mirror geometry. Under local mirror regularity, a quantitative local Bregman growth condition, and bounded Brownian diffusion, we establish finite-horizon local terminal-time bounds in expectation and with high probability, together with an exit-probability estimate for the localization neighborhood. The result provides a unified variational construction of the induced predictive-memory mirror flow together with a local stochastic certificate for last-iterate performance near stable equilibria.
\end{abstract}

\section{Introduction}
{Learning in games studies equilibrium seeking through players' iterative strategy updates based on environmental feedback \cite{tao22info}, with applications in multi-agent systems \cite{tao22confluence}. Among these schemes, mirror descent generalizes gradient descent by linking primal strategy iterates and dual gradients through a mirror map tailored to the domain geometry \cite{nemirovsky83md}. Its ergodic convergence to Nash equilibria has been studied extensively \cite{nemirovski04prox, hsieh21optmd}. However, average convergence does not imply satisfactory \emph{last-iterate} behavior. In weakly stable games, mirror descent dynamics may exhibit recurrent or rotational transients, sometimes known as Poincar\'{e} recurrence \cite{mertikopoulos2018cycles}, and stochastic perturbations further degrade pointwise performance.}

To mitigate cycling behavior, recent advances introduce predictive modifications: each gradient update is no longer based on the realized feedback from the last step; instead, a prediction of unrealized feedback is used. Examples include optimistic mirror descent \cite{hsieh21optmd}, extra-gradient \cite{mertikopoulos19peg}, mirror-prox \cite{nemirovski04prox}, and more generally, online learning with predictive sequences \cite{chiang12gradual-var, rakhlin13predictable}. From a mirror descent perspective, while the aforementioned algorithms appear in distinct forms, they all apply the mirror map to the prediction of the current, unrealized feedback at each step.    

{However, predictive mirror descent (\textsc{pmd}) still requires structure in the prediction channel. As pointed out in \cite{hsieh21optmd}, \textsc{pmd} may fail last-iterate convergence if the prediction term is not properly configured. This motivates a control-synthesis viewpoint: prediction is generated by an auxiliary memory state coupled to the mirror dynamics, with the feedback law determined by a variational principle. We develop this formulation for convex games with stochastic gradient feedback.}

{Specifically, \textbf{our contributions are threefold}. We formulate a stochastic mirror differential game (\textsc{smdg}) that augments mirror play with predictive strategic and realized corrective channels through an auxiliary memory state. We show in Thm.~\ref{thm:mirror-path} that its closed-loop Nash equilibrium synthesizes predictive mirror descent dynamics. Under local regularity and growth conditions, we establish finite-horizon terminal-time bounds in expectation (Thm.~\ref{thm:expected-terminal}) and with high probability (Thm.~\ref{thm:high-prob-terminal}) for trajectories that remain in the localization neighborhood, together with an explicit exit-probability estimate.}

\noindent\textbf{Related Works.}
Mirror-based learning and regularized game dynamics have been studied extensively in both discrete and continuous time \cite{PM16lyapnov,mertikopoulos2019learning,gao2020continuous,gao2022continuous,pan-tao22noneq, pan-tao23delay, tao26doco}. On the algorithmic side, extra-gradient\cite{mertikopoulos19peg}, mirror-prox \cite{nemirovski04prox}, and optimistic mirror descent \cite{hsieh21optmd} provide predictive modifications that improve stability and last-iterate behavior in monotone variational inequalities, saddle-point problems, and games \cite{tseng2000modified,rakhlin13predictable,daskalakis2017training,pan-tao23delay}. These results are mostly discrete-time and algorithmic; {in particular, \cite{chiang12} first explored optimism in mirror descent, which later developed into general predictive mirror (gradient) descent learning \cite{Rakhlin13predictive}, where the learning agent queries an oracle for predicted gradient information. Following these early efforts, Mertikopoulos et al. carried out a comprehensive examination on the convergence of discrete-time optimistic mirror descent \cite{hsieh21optmd}, mirror-prox \cite{mertikopoulos2018optimistic}, and extra-gradient \cite{mertikopoulos19peg}.  These methods motivate the predictive-query template below; our contribution is a continuous-time synthesis result: we construct the predictive law through a differential game and analyze finite-horizon terminal-time performance under stochastic feedback.}

Variational viewpoints on optimization and learning provide the second strand of related work \cite{tzen2023variational}. Previous work \cite{pan2024variational} used this perspective to represent vanilla mirror play. The present paper keeps that differential-game backbone but shifts from representation to design: the predictive law is synthesized through a two-channel Fenchel structure, and the resulting stochastic dynamics admit expectation and high-probability terminal-time bounds near a stable equilibrium.

\section{Problem Formulation}
{\noindent\textbf{Preliminaries and Notation.}}
 Let $\mathscr{E}=\mathbb{R}^n$ be endowed
with the Euclidean inner product $\langle x,y\rangle\triangleq x^\top y$ and the associated norm $\|\cdot\|$.
For a proper closed convex function $f:\mathscr{E}\to(-\infty,+\infty]$, its Fenchel conjugate is
$f^*(y)\triangleq\sup_{x\in\mathscr{E}}\{\langle x,y\rangle-f(x)\}$ and the associated Fenchel coupling is
$\mathcal{FC}_f(x,y)\triangleq f(x)+f^*(y)-\langle x,y\rangle$. If $f$ is differentiable, the induced Bregman
divergence is $D_f(x,y)\triangleq f(x)-f(y)-\langle \nabla f(y),\,x-y\rangle$. If $f$ is Legendre, then
$\nabla f$ is bijective with inverse $\nabla f^*$ and
$D_{f^*}(\nabla f(x),\nabla f(y))=D_f(y,x)$. An operator $\Psi:\mathcal{X}\to\mathscr{E}$
is \emph{$L$-Lipschitz} if $\|\Psi(x)-\Psi(x')\|\le L\|x-x'\|$ for all $x,x'\in\mathcal{X}$. Throughout the discussion, the capitalized letter $X$ denotes the random variable, while the uncapitalized one denotes one realization.

{\noindent\textbf{Convex Game.} Consider a continuous-kernel game defined by the tuple $\mathcal{G}=(\mathcal{N}, \{\mathcal{Y}_i\}_{i\in \mathcal{N}}, \{\psi_i\}_{i\in \mathcal{N}})$, where $\mathcal{N}$ is the player set and $\mathcal{Y}_i\subset \R^{n_i}$ is the strategy set of the $i$-th player, assumed to be nonempty, closed, and convex. Under a joint strategy $y=(y_i, y_{-i})\in \mathcal{Y}\triangleq \prod_{i\in \mathcal{N}}\mathcal{Y}_i\subset \R^{n}, n=\sum_{i\in \mathcal{N}}n_i$, player $i$'s cost is determined by the twice continuously differentiable function $\psi_i(y_i, y_{-i})$, which is convex in $y_i$ for fixed $y_{-i}$. Denote by $\nabla_i \psi_i(y_i, y_{-i})$ the partial derivative with respect to $y_i$ and define the \textit{pseudo-gradient} operator as $\Psi(y)=\operatorname{col}(\nabla_i \psi_i(y_i, y_{-i}))_{i\in \mathcal{N}}$. A profile $y^\star=(y^\star_i, y_{-i}^\star)\in \mathcal{Y}$ is a Nash equilibrium (\textsc{ne}) if no player can reduce its cost by deviating unilaterally. Under the stated convexity and differentiability conditions, this no-deviation condition is equivalent to the variational inequality}
\begin{equation}
\label{eq:vi}
    {\langle \Psi(y^\star), y-y^\star \rangle\geq 0, \quad \forall y\in \mathcal{Y}.}
\end{equation}

\noindent\textbf{Predictive Mirror Descent.} Mirror descent (\textsc{md}) refers to a class of first-order methods tailored to the problem geometry via a Legendre potential function \cite{nemirovsky83md}. Suppose each player $i$ is equipped with such a potential function $\phi_i: \mathcal{Y}_i\rightarrow\R$, whose Fenchel conjugate is denoted by $\phi_i^*$. With stepsize $h>0$, player $i$'s mirror descent step is
\begin{equation}
\label{eq:md-discrete}
y_i^t=\argmin_{y_i\in \mathcal{Y}_i} \{ h\langle \nabla_i\psi_i(y^{t-1}), y_i \rangle + D_{\phi_i}(y_i, y_{i}^{t-1})\},
\end{equation}
which admits the dual-space representation under the mirror map $\nabla \phi_i: \mathcal{Y}_i\rightarrow \mathcal{X}_i$ and $\nabla \phi^*_i: \mathcal{X}_i\rightarrow \mathcal{Y}_i$:
$x_i^t=x_i^{t-1}-h\nabla_i\psi_i(\nabla\phi_i^*(x_i^{t-1}), \nabla\phi_{-i}^*(x_{-i}^{t-1})).$
The leading continuous-time limit of the concatenated update is
\begin{equation}
\label{eq:md-cont}
        \dot{x}(t)=-\Psi(\nabla\phi^*(x(t))),\quad  y(t)  = \nabla\phi^*(x(t)),
\end{equation}
where $x(t)$ and $y(t)$ are within the dual and primal space, respectively;  $\Psi(y)\triangleq \operatorname{col}(\nabla_i\psi_i(y_i, y_{-i}))_{i\in \mathcal{N}}$ and $\nabla\phi^*(x)\triangleq \operatorname{col}(\nabla\phi^*_i(x_i))_{i\in \mathcal{N}}$.
Even in static games, the mirror flow can exhibit recurrent trajectories rather than last-iterate convergence \cite{mertikopoulos2018cycles}, motivating predictive queries.

A broad class of predictive mirror methods modifies this update by first constructing a method-dependent auxiliary query point $\hat x^t$ in the dual space and then evaluating the pseudo-gradient at the associated profile $\hat y^t$. At this stage, the rule generating $\hat x^t$ is left open; the stochastic mirror differential game below will generate a particular such query endogenously.
\begin{subequations}
    \label{eq:pred-md-discrete}
\begin{align}
\hat y_i^t&=\nabla\phi_i^*(\hat x_i^t), \\
x_i^t&=x_i^{t-1}-h\nabla_i\psi_i(\hat y^t), \qquad y_i^t=\nabla\phi_i^*(x_i^t).
\end{align}
\end{subequations}
Mirror-prox/extra-gradient uses a one-step query correction, while optimistic mirror descent uses a history-based gradient extrapolation. These mechanisms are not identical algorithms, but they share the same predictive-query principle: the mirror update is driven by feedback at an auxiliary profile rather than only at the realized iterate.

This commonality is also visible in their small-stepsize expansions. For mirror-prox/extra-gradient, the correction arises by expanding the intermediate query $\hat x=x-h\Psi(y)$ and $\hat y=\nabla\phi^*(\hat x)$; optimistic mirror descent yields the same term through gradient extrapolation. Although the ordinary $h\to0$ limit of these methods is the vanilla \textsc{md} flow \eqref{eq:md-cont}, their first modified equations contain a predictive correction of the form
$
\dot x
=-\Psi(y)+h\,D\Psi(y)\nabla^2\phi^*(x)\Psi(y)+O(h^2),
y=\nabla\phi^*(x).$
Thus, the stabilizing effect of prediction is a next-order correction rather than a different leading mirror flow. This motivates the construction below, where the auxiliary query is generated endogenously through a stochastic mirror differential game.

{\noindent\textbf{Stochastic Mirror Differential Game.} We approach \textsc{pmd} through variational optimal control, drawing on the Brezis--Ekeland principle \cite{brezis1976principe}, which characterizes gradient descent in convex programming through an associated optimal control problem. This viewpoint was later extended to single-agent \textsc{md} \cite{tzen2023variational} and multi-agent \textsc{md} in monotone games \cite{pan2024variational}.}

{Leveraging the Brezis-Ekeland principle, we construct a stochastic differential game below and prove in Sec.~\ref{sec:variation} that its closed-loop equilibrium induces a concrete realization of the predictive mirror template, with the auxiliary memory dynamics specified by the corrective channel.}
{Because $\mathcal G$ is static, its reference equilibrium $y^\star \in \mathcal Y$ is time-independent; the horizon $[0,T]$ belongs to the auxiliary differential game. We augment the stochastic state as $(Z_t=(X_t, E_t)\in\mathcal X\times\mathcal X)$, where $X_t$ is the concatenated mirror dual state of all players and $E_t$ is the auxiliary prediction state.} The predictive dual state is 
$\hat X_t\triangleq X_t+E_t$. Under the mirror map, $Y_t=\nabla\phi^*(X_t)$ and $\hat Y_t=\nabla\phi^*(\hat X_t)$. Our analysis rests on a complete filtered probability space
$(\Omega,\mathcal{F},\mathbb{F},\mathbb{P})$ with filtration
$\mathbb{F}\triangleq\{\mathcal{F}_t\}_{t\ge 0}$ satisfying the usual conditions \cite{pan2024variational}.
All state trajectories, Brownian motions, and feedback controls are assumed to be $\mathbb{F}$-adapted.

Each player chooses a strategic control $U_{i,t}\in\mathcal X_i$ and a corrective control $V_{i, t}\in\mathcal X_i$ that jointly control the following stochastic dynamics
\begin{equation}
\label{eq:spmdg-dynamics}
\begin{aligned}
\mathrm{d}X_{i,t} &= U_{i,t}\mathrm{d}t + \Sigma_i(Z_t)\mathrm{d}W_{i,t},\\
\mathrm{d}E_{i,t} &= V_{i,t}\mathrm{d}t + \widetilde\Sigma_i(Z_t)\mathrm{d}\widetilde W_{i,t},
\end{aligned}
\qquad i\in\mathcal N,
\end{equation}
{where $\{W_i,\widetilde W_i\}_{i\in\mathcal N}$ are mutually independent standard Brownian motions and the diffusion matrices $\Sigma_i$, $\widetilde\Sigma_i$ are measurable maps.}

For any fixed profile $y_{-i}\in\mathcal{Y}_{-i}$, define the partial Fenchel conjugate of player $i$'s {strategic} control $u_i$ by
$\psi_i^*(u_i\mid y_{-i})
\triangleq
\sup_{y_i\in\mathcal Y_i}
\big\{\langle u_i,y_i\rangle-\psi_i(y_i,y_{-i})\big\}$. Consider the Fenchel coupling (see Preliminaries) between the predictive state and the control:
\begin{align*}
\mathcal{FC}_{\psi_i(\cdot,\hat{y}_{-i})}(\hat{y}_i, -u_i)
&= \psi_i(\hat y_i,\hat y_{-i})
+ \psi_i^*(-u_i\mid \hat y_{-i})\\
&\quad + \langle u_i,\hat{y}_i\rangle,
\end{align*}
{which is nonnegative by the Fenchel--Young inequality and attains zero if and only if $u_i=-\nabla_i\psi_i(\hat{y})$.}

{We exploit this property and design the game stage cost by enforcing the Fenchel principle on both the strategic-control and corrective-control channels:}
\begin{equation}
\label{eq:spmdg-stage-cost}
\begin{aligned}
c_i(z,& u,v)
=
\psi_i(\hat y_i,\hat y_{-i})
+ \psi_i^*(-u_i\mid \hat y_{-i})
+ \langle u_i,y_i^\star\rangle \\
 + \alpha\eta\, & 
\mathcal{FC}_{\psi_i(\cdot,y_{-i})}
\!\left(
 y_i,
 -\frac{v_i+\alpha e_i}{\alpha\eta}
\right) - \langle v_i,\hat y_i-y_i^\star\rangle,
\end{aligned}
\end{equation}
{where $z=(x,e)$, $y=\nabla\phi^*(x)$, $\hat y=\nabla\phi^*(x+e)$, and $\alpha,\eta>0$ are design parameters. Intuitively, the strategic channel drives the mirror state along the pseudo-gradient evaluated at the predicted profile, while the corrective channel updates the auxiliary state $e_i$ through a memory filter with relaxation rate $\alpha$ and realized-gradient scale $\eta$. In equilibrium, this filter drives $e_i$ toward $-\eta\nabla_i\psi_i(y)$, so the predictive state $\hat x_i=x_i+e_i$ is a displaced dual variable whose magnitude is set by $\eta$ and whose relaxation speed is set by $\alpha$.}

{For the same reference equilibrium $y^\star$, the terminal cost of player $i$ is defined on the augmented dual state: $q_i(Z_T)=D_{\phi_i^*}(\hat X_{i,T},x_i^\star)$, $x_i^\star=\nabla\phi_i(y_i^\star)$. Accordingly, player $i$ seeks to minimize the expected cumulative cost under joint controls ${U}\triangleq \{U_t=(U_{i,t})_{i\in \mathcal{N}}\}_{t\in[0,T]}$, ${V}\triangleq \{V_t=(V_{i,t})_{i\in \mathcal{N}}\}_{t\in[0,T]}$}
\begin{equation}
\label{eq:spmdg-cost}
J_i(z)
=\mathbb E\left[\int_0^T c_i(Z_t,U_t,V_t)\mathrm{d}t + q_i(Z_T)\middle|Z_0=z\right].
\end{equation}

\section{Equilibrium Mirror Path and Local Stochastic Stability}
\label{sec:variation}

{The auxiliary game is useful only insofar as its equilibrium feedback determines an actual mirror trajectory. We first compute this trajectory and then ask whether it stays close to $y^\star$ when the feedback is noisy. Locally, the same Bregman energy that defines the mirror geometry serves as the stability certificate. Relative to the variational representation of vanilla mirror play in \cite{pan2024variational}, the corrective Fenchel channel generates the predictive memory state.}

\subsection{Equilibrium feedback synthesis}
\label{sec:equilibrium-synthesis}

Fix a reference equilibrium $y^\star\in\mathcal Y$ and let
$x^\star=\nabla\phi(y^\star)$.
For player $i\in\mathcal N$, consider the candidate predictive-state {value ansatz}
$V_i(z)\triangleq D_{\phi_i^*}(x_i+e_i,x_i^\star)$ for $z=(x,e)$.
Since $\hat x_i=x_i+e_i$ and $\hat y_i=\nabla\phi_i^*(\hat x_i)$,
its gradients with respect to both control channels coincide:
$\nabla_{x_i}V_i(z)=\hat y_i-y_i^\star$ and
$\nabla_{e_i}V_i(z)=\hat y_i-y_i^\star$. Hence, the control-dependent part of player $i$'s Hamiltonian is
$\langle \hat y_i-y_i^\star,u_i\rangle + \langle \hat y_i-y_i^\star,v_i\rangle + c_i(z,u_i,v_i)$.
Substituting \eqref{eq:spmdg-stage-cost} and collecting the control-dependent terms yields the following decomposition.

\begin{lemma}
\label{lem:two-fenchel-hamiltonian}
For every player $i\in\mathcal N$ and state $z=(x,e)$, the control-dependent Hamiltonian equals
\begin{equation}
\label{eq:control-hamiltonian}
\begin{aligned}
\mathscr H_i(z,u_i,v_i)
&= 
\mathcal{FC}_{\psi_i(\cdot,\hat y_{-i})}(\hat y_i,-u_i) \\
\quad + \alpha\eta\,
& \mathcal{FC}_{\psi_i(\cdot,y_{-i})}
\!\left(
 y_i,
 -\frac{v_i+\alpha e_i}{\alpha\eta}
\right) 
 + \alpha\langle e_i,y_i\rangle,
\end{aligned}
\end{equation}
where the final term is independent of $(u_i,v_i)$. 

Fix $i\in\mathcal N$ and a state $z=(x,e)$. Suppose that for every fixed $\zeta_{-i}\in\mathcal Y_{-i}$, the map
$y_i\mapsto \psi_i(y_i,\zeta_{-i})$ is proper, closed, convex, and differentiable on
$\operatorname{int}(\dom(\phi_i))$.
Then the pointwise minimizers of \eqref{eq:control-hamiltonian} satisfy
\begin{equation}
\label{eq:u-v-star}
u_i^\star = -\nabla_{y_i}\psi_i(\hat y_i,\hat y_{-i}),
\quad
v_i^\star = -\alpha e_i-\alpha\eta\nabla_{y_i} \psi_i(y).
\end{equation}
\end{lemma}

\begin{proof}
For the strategic channel,
$\langle \hat y_i-y_i^\star,u_i\rangle
+\psi_i(\hat y_i,\hat y_{-i})
+\psi_i^*(-u_i\mid \hat y_{-i})
+\langle u_i,y_i^\star\rangle
=\mathcal{FC}_{\psi_i(\cdot,\hat y_{-i})}(\hat y_i,-u_i)$.
For the corrective channel, the co-state correction cancels the $v_i$ shadow price induced by the predictive-state value because
$\langle \hat y_i-y_i^\star,v_i\rangle-\langle v_i,\hat y_i-y_i^\star\rangle=0$.
Therefore, the remaining $v_i$-dependence is exactly
$\alpha\eta\,\mathcal{FC}_{\psi_i(\cdot,y_{-i})}\!\left(y_i,-\frac{v_i+\alpha e_i}{\alpha\eta}\right)$,
and expanding this Fenchel coupling contributes the state-only term $\alpha\langle e_i,y_i\rangle$.

By Fenchel--Young inequality,
$\mathcal{FC}_{\psi_i(\cdot,\hat y_{-i})}(\hat y_i,-u_i)\ge 0$
with equality if and only if
$-u_i=\nabla_{y_i}\psi_i(\hat y_i,\hat y_{-i})$.
Likewise,
$\mathcal{FC}_{\psi_i(\cdot,y_{-i})}\!\left(y_i,-\frac{v_i+\alpha e_i}{\alpha\eta}\right)\ge 0$
with equality if and only if
$-\frac{v_i+\alpha e_i}{\alpha\eta}=\nabla_{y_i}\psi_i(y_i,y_{-i})$.
Since the final term in \eqref{eq:control-hamiltonian} is independent of $(u_i,v_i)$, the minimizers are exactly those that make both Fenchel couplings vanish, which yields \eqref{eq:u-v-star}.
\end{proof}

\begin{theorem}
\label{thm:mirror-path}
Consider the \textsc{smdg} defined by \eqref{eq:spmdg-dynamics}--\eqref{eq:spmdg-cost}.
Let $(U^\star,V^\star)$ be a closed-loop equilibrium profile whose Hamiltonian minimization over $(U_i,V_i)$ is attained pointwise for every player relative to the predictive-state value ansatz $V_i(z)=D_{\phi_i^*}(x_i+e_i,x_i^\star)$.
Then the equilibrium feedback laws satisfy
\begin{equation}
\label{eq:equilibrium-feedback}
U_{i,t}^\star = -\nabla_{y_i}\psi_i(\hat Y_t^\star),
\ \
V_{i,t}^\star = -\alpha E_{i,t}^\star-\alpha\eta\nabla_{y_i}\psi_i(Y_t^\star),
\end{equation}
and the induced equilibrium trajectory $Z_t^\star=(X_t^\star,E_t^\star)$ evolves according to
\begin{equation}
\label{eq:pmp}
\begin{aligned}
dX_t^\star &= -\Psi(\hat Y_t^\star)dt + \Sigma(Z_t^\star)dW_t,\\
dE_t^\star &= \big[-\alpha E_t^\star-\alpha\eta\Psi(Y_t^\star)\big]dt + \widetilde\Sigma(Z_t^\star)d\widetilde W_t,
\end{aligned}
\end{equation}
where
$Y_t^\star=\nabla\phi^*(X_t^\star)$ and $\hat Y_t^\star=\nabla\phi^*(X_t^\star+E_t^\star)$.
\end{theorem}

\begin{proof}[Proof Sketch]
{The theorem is the closed-loop realization of the two-channel Hamiltonian synthesis developed above. We follow the same verification logic used in \cite{pan2024variational}. For simplicity, we omit the full viscosity solution characterization here. By Lemma~\ref{lem:two-fenchel-hamiltonian}, the control-dependent Hamiltonian separates into a predictive strategic Fenchel coupling in $u_i$ and a realized corrective Fenchel coupling in $v_i$, up to a term independent of the controls. We can then identify the unique pointwise minimizers:}
$
u_i^\star=-\nabla_{y_i}\psi_i(\hat y),
v_i^\star=-\alpha e_i-\alpha\eta\nabla_{y_i}\psi_i(y).$
Since Hamiltonian minimization is attained pointwise along the equilibrium profile, these minimizers coincide with the equilibrium feedback laws, yielding \eqref{eq:equilibrium-feedback}. 

Substituting these feedback laws into the controlled state equations \eqref{eq:spmdg-dynamics} gives \eqref{eq:pmp}.
\end{proof}

\subsection{Local stability assumptions and stopped process}

We now analyze the equilibrium mirror path locally around the equilibrium $y^\star$.
Define the Lyapunov energy
\begin{equation}
\label{eq:lyap}
\mathcal V(x,e)
\triangleq 
D_{\phi^*}(x,x^\star)+\frac{\kappa}{2}\|e\|^2,
\qquad \kappa>0.
\end{equation}
{Here $\kappa$ simply fixes the relative scale between the mirror-state error and the auxiliary-memory error.}

\begin{assumption}[Local mirror regularity, Lipschitzness, and bounded diffusion]
\label{ass:local-reg}
There exist neighborhoods $\mathcal X_\star\subseteq \mathcal X$ of $x^\star=\nabla\phi(y^\star)$ and $\mathcal U_\star\subseteq \mathcal Y$ of $y^\star$, and constants
$m_\phi,M_\phi,L_\Psi,B_\Psi,\sigma_X,\sigma_E>0$, such that 
\begin{align}
\label{eq:phi-local-bounds}
m_\phi I \preceq \nabla^2\phi^* & (x)  \preceq M_\phi I,
&&\forall x\in \mathcal X_\star,
\\
\label{eq:psi-local-bounds}
\|\Psi(y)-\Psi(y')\| &\le L_\Psi\|y-y'\|,
&&\forall y,y'\in \mathcal U_\star,
\end{align}
and whenever $x\in\mathcal X_\star$ and $x+e\in\mathcal X_\star$, the dual states
$y=\nabla\phi^*(x)$ and $\hat y=\nabla\phi^*(x+e)$ belong to $\mathcal U_\star$, while the diffusion coefficients satisfy
$\|\Sigma(x,e)\|_F \le \sigma_X$, $\|\widetilde\Sigma(x,e)\|_F \le \sigma_E$, and $\|\Psi(y)\| \le B_\Psi$.
\end{assumption}

\begin{assumption}[Local Bregman growth condition]
\label{ass:local-vs}
There exists a constant $c_{\mathrm{vs}}>0$ such that for every $x\in\mathcal X_\star$ with
$y=\nabla\phi^*(x)\in\mathcal U_\star$,
\begin{equation}
\label{eq:local-vs}
\langle \Psi(y),y-y^\star\rangle
\ge
c_{\mathrm{vs}}\,D_{\phi^*}(x,x^\star).
\end{equation}
\end{assumption}

{Assumption~\ref{ass:local-vs} gives a local error-growth scale: near $y^\star$, the variational gap controls the dual Bregman distance to equilibrium. This is the coercivity used in the drift estimate. Note that such an assumption is a direct extension of the second-order sufficiency in variational stability, e.g., see \cite[Assumption 2]{mertikopoulos21opt-md}, and the second-order test in \cite[Assumption 4]{shutian23erm}, where the Euclidean distance is replaced by the Bregmann divergence.} For $x,x+e\in\mathcal X_\star$, let $C_\phi\triangleq {2M_\phi^2}/{m_\phi}$; then
\begin{equation}
\label{eq:y-d-bound}
\|y-y^\star\|^2 \le C_\phi D_{\phi^*}(x,x^\star),\;
\|\hat y-y\| \le M_\phi\|e\|, 
\end{equation}

{Fix $r>0$ and define the sublevel set
$\mathcal K_r \triangleq  \big\{(x,e): \mathcal V(x,e)\le r\big\}$.
Choose $r$ so that $x\in\mathcal X_\star$ and $x+e\in\mathcal X_\star$ whenever $(x,e)\in\mathcal K_r$.
Define the exit time $\tau_r \triangleq  \inf\big\{t\ge 0 : (X_t,E_t)\notin \mathcal K_r\big\}$, with the convention $\inf\varnothing=+\infty$. We use the stopped process
$Z_{t\wedge\tau_r}=(X_{t\wedge\tau_r},E_{t\wedge\tau_r})$ to localize the drift calculation. The resulting bounds apply to trajectories that remain in $\mathcal K_r$ over the finite horizon, while Theorem~\ref{thm:expected-terminal} accounts for the exit probability.}

\subsection{Local drift, expected terminal error, and concentration}

\begin{lemma}[Local It\^o-Lyapunov drift]
\label{lem:local-drift}
Under Assumptions~\ref{ass:local-reg} and \ref{ass:local-vs}, {suppose the Lyapunov weight $\kappa>0$ satisfies}
\begin{equation}
\label{eq:drift-feasibility}
\kappa\alpha>\frac{C_{\mathrm{mix}}^2}{2c_{\mathrm{vs}}},
\qquad
C_{\mathrm{mix}}\triangleq L_\Psi M_\phi\sqrt{C_\phi}.
\end{equation}
For any parameter
$
\mu\in\left(\frac{1}{\kappa\alpha},\frac{2c_{\mathrm{vs}}}{C_{\mathrm{mix}}^2}\right)
$
define
\begin{equation}
\label{eq:drift-constants}
\begin{aligned}
c_D \triangleq  c_{\mathrm{vs}}-\frac{\mu C_{\mathrm{mix}}^2}{2}>0, \quad & 
c_E \triangleq  \frac{\kappa\alpha}{2}-\frac{1}{2\mu}>0.
\end{aligned}
\end{equation}
Then on $[0,\tau_r]$ the infinitesimal generator $\mathcal L$ of $\mathcal V$ satisfies
\begin{equation}
\label{eq:local-drift-split}
\mathcal L\mathcal V(X_t,E_t)
\le
-c_D D_{\phi^*}(X_t,x^\star)-c_E\|E_t\|^2+b \quad a.s.,
\end{equation}
where
$ 
b \triangleq  \frac{M_\phi}{2}\sigma_X^2 + \frac{\kappa}{2}\sigma_E^2 + \frac{\kappa\alpha\eta^2 B_\Psi^2}{2}.
$
Consequently, for $\lambda\triangleq \min\left\{c_D,\frac{2c_E}{\kappa}\right\}$,
\begin{equation}
\label{eq:local-drift-main}
\begin{aligned}
\mathcal L\mathcal V(X_t,E_t)
\le
-\lambda\mathcal V(X_t,E_t)+b.
\end{aligned}
\end{equation}
\end{lemma}

\begin{proof}
Apply It\^o's formula to $D_{\phi^*}(X_t,x^\star)$ and $\frac{\kappa}{2}\|E_t\|^2$ along \eqref{eq:pmp}. The realized-state Bregman term contributes the drift
$-\langle Y_t-y^\star,\Psi(\hat Y_t)\rangle,$
which we split as
$
-\langle Y_t-y^\star,\Psi(Y_t)\rangle
-\langle Y_t-y^\star,\Psi(\hat Y_t)-\Psi(Y_t)\rangle.$

Assumption~\ref{ass:local-vs} controls the first term by
$-c_{\mathrm{vs}}D_{\phi^*}(X_t,x^\star)$.
The second term is bounded using \eqref{eq:y-d-bound}, the local Lipschitz constant $L_\Psi$, and Young's inequality, which yields the mixed coefficients in \eqref{eq:drift-constants}. For the corrective channel, the drift of $\frac{\kappa}{2}\|E_t\|^2$ is
$
-\kappa\alpha\|E_t\|^2-\kappa\alpha\eta\langle E_t,\Psi(Y_t)\rangle,$
and the last term is bounded by Young's inequality and the local bound $\|\Psi(Y_t)\|\le B_\Psi$. The feasibility condition \eqref{eq:drift-feasibility} guarantees that the interval
$
\left(\frac{1}{\kappa\alpha},\frac{2c_{\mathrm{vs}}}{C_{\mathrm{mix}}^2}\right)$
is nonempty, so one can choose $\mu$ such that both $c_D$ and $c_E$ are positive. The diffusion trace terms are bounded using Assumption~\ref{ass:local-reg}. The detailed coefficient bookkeeping is recorded in Appendix~\ref{app:drift-proof}.
\end{proof}

\begin{theorem}
\label{thm:expected-terminal}
Under the assumptions of Lemma~\ref{lem:local-drift}, define
\[
A_T\triangleq e^{-\lambda T}\mathcal V(Z_0)
+
\frac{b}{\lambda}\big(1-e^{-\lambda T}\big).
\]
Then, for every horizon $T\ge 0$,
\begin{equation}
\label{eq:expected-terminal-main}
\mathbb E\big[\mathcal V(Z_T)\mathbf 1_{\{\tau_r>T\}}\big]
\le A_T.
\end{equation}
In particular,
\begin{align}
&\mathbb E\big[D_{\phi^*}(X_T,x^\star)\mathbf 1_{\{\tau_r>T\}}\big]
\le A_T, \label{eq:expected-terminal-bregman} \\
& \mathbb E\big[\|E_T\|^2\mathbf 1_{\{\tau_r>T\}}\big]
\le \frac{2A_T}{\kappa}.\label{eq:expected-terminal-error}
\end{align}
Moreover,
\begin{equation}
\label{eq:exit-prob}
\mathbb P(\tau_r\le T)
\le \varepsilon_T(r)
\triangleq \min\left\{1,\frac{\mathcal V(Z_0)+bT}{r}\right\}.
\end{equation}
\end{theorem}

\begin{proof}[Proof Sketch]
Apply It\^o's formula to $e^{\lambda(t\wedge\tau_r)}\mathcal V(Z_{t\wedge\tau_r})$ and take expectations. Restricting the result to $\{\tau_r>T\}$ gives \eqref{eq:expected-terminal-main}; the componentwise bounds follow from the definition of $\mathcal V$. For \eqref{eq:exit-prob}, apply the unweighted stopped It\^o estimate, drop the dissipative term, and use continuity:
\[
r\,\mathbb P(\tau_r\le T)
\le \mathbb E[\mathcal V(Z_{T\wedge\tau_r})]
\le \mathcal V(Z_0)+bT.
\]
\end{proof}

\begin{theorem}
\label{thm:high-prob-terminal}
Under the assumptions of Lemma~\ref{lem:local-drift}, define $\nu_r^2 \triangleq  r\big(C_\phi\sigma_X^2 + 2\kappa\sigma_E^2\big)$.
Then for every $T\ge 0$ and every $\delta\in(0,1)$, with probability at least $1-\delta-\varepsilon_T(r)$,
\begin{equation}
\label{eq:high-prob-main}
\begin{aligned}
\mathcal V(Z_T)
&\le e^{-\lambda T}\mathcal V(Z_0)
+ \frac{b}{\lambda}(1-e^{-\lambda T}) \\
&\quad + \nu_r\sqrt{\frac{1-e^{-2\lambda T}}{\lambda}\log\frac{1}{\delta}}.
\end{aligned}
\end{equation}
Consequently, the same bound holds with $\mathcal V(Z_T)$ replaced by
$D_{\phi^*}(X_T,x^\star)$.
\end{theorem}

\begin{proof}[Proof Sketch]
On the event $\{\tau_r>T\}$, use the drift decomposition from Lemma~\ref{lem:local-drift} and apply the integrating factor $e^{\lambda t}$:
\[
\mathcal V(Z_T)
\le
 e^{-\lambda T}\mathcal V(Z_0)
+
\frac{b}{\lambda}(1-e^{-\lambda T})
+
N_T,
\]
where $N_T$ is the exponentially weighted martingale term. On the stopped region $\mathcal K_r$, the local bounds on $D_{\phi^*}(X_t,x^\star)$, $\|E_t\|^2$, and the diffusion coefficients imply the deterministic quadratic-variation estimate
\[
[N]_T
\le
\nu_r^2\int_0^T e^{-2\lambda(T-s)}ds.
\]
A continuous-time exponential martingale inequality \cite{karatzas1991brownian}, combined with \eqref{eq:exit-prob}, then yields \eqref{eq:high-prob-main}. The details are given in Appendix~\ref{app:high-prob-proof}.
\end{proof}

\begin{remark}[Deterministic limit]
\label{rem:deterministic-limit}
{If the diffusion channels vanish identically on the local neighborhood, then $b$ loses the diffusion terms and the martingale term disappears. In that case, on $\{\tau_r>T\}$,}
\[
\mathcal V(Z_T)
\le
 e^{-\lambda T}\mathcal V(Z_0)
+
\frac{\kappa\alpha\eta^2 B_\Psi^2}{2\lambda}(1-e^{-\lambda T}).
\]
{The residual floor comes from the crude local bound $\|\Psi(Y_t)\|\le B_\Psi$. If one additionally has $\Psi(y^\star)=0$ (e.g., $y^\star$ is an interior equilibrium), then Appendix~\ref{app:drift-proof} shows that the cross term can be absorbed into the realized-state and error dissipation terms, so $b$ reduces to the diffusion-only contribution. In particular, the deterministic path then enjoys exact exponential decay within the local neighborhood for sufficiently small $\eta$.}
\end{remark}

\subsection{Numerical illustration}
\label{sec:numerical-illustration}

{We apply the Euler discretization of \eqref{eq:pmp} to the biased matching-pennies game of \cite{mertikopoulos2018cycles} (player-1 payoff $\bigl[\begin{smallmatrix}1&-1\\-1&8\end{smallmatrix}\bigr]$, equilibrium $y^\star=(9/11,9/11)$, entropic geometry):
$e^{k+1}=(1-h\alpha)e^k-h\alpha\eta\Psi(y^k)$ and
$x^{k+1}=x^k-h\Psi(\nabla\phi^*(x^k+e^k))$.
Fig.~\ref{fig:numerical-comparison} compares this Euler-PMDG update with mirror descent, optimistic mirror descent, and extra-gradient/mirror-prox under noisy gradient queries, and shows the sensitivity of its terminal error to $(\alpha,\eta)$ over a log-spaced parameter grid. 
}

\begin{figure}[t]
\centering
\includegraphics[width=\linewidth]{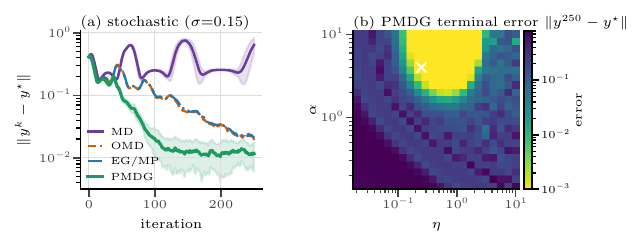}
\caption{{Biased matching pennies \cite{mertikopoulos2018cycles}, $h=0.10$, 250 steps. (a) Mean last-iterate error over 30 noisy trials ($\sigma=0.15$); shading shows one standard deviation for MD and PMDG. (b) Deterministic PMDG terminal error over $(\alpha,\eta)$; the cross marks the panel~(a) operating point.}}
\label{fig:numerical-comparison}
\end{figure}

\section{Conclusion}

This letter treats predictive mirror descent in convex games as a stochastic control-synthesis problem. The proposed stochastic mirror differential game augments mirror play with predictive strategic and realized corrective channels, tied together through a two-channel Fenchel construction that preserves the variational backbone of mirror play while generating the predictive law from the game itself.

{The resulting guarantees are local and finite-horizon: near a stable equilibrium, the localized predictive mirror path satisfies a stochastic Lyapunov estimate, yielding terminal-time bounds together with an exit-probability estimate. The next step is to extend this design perspective beyond the local regime while weakening the growth assumptions needed for concentration.}

\bibliographystyle{ieeetr}
\bibliography{ref}

\appendix

\subsection{Proof for Lemma~\ref{lem:local-drift}}
\label{app:drift-proof}

We compute on $[0,\tau_r]$ and then invoke standard localization for the stopped process.
Set
$Y_t \triangleq  \nabla\phi^*(X_t)$, $\hat Y_t \triangleq  \nabla\phi^*(X_t+E_t)$,
$\Sigma_t \triangleq  \Sigma(X_t,E_t)$, and $\widetilde\Sigma_t \triangleq  \widetilde\Sigma(X_t,E_t)$.
By \eqref{eq:pmp},
$dX_t = -\Psi(\hat Y_t)dt + \Sigma_t dW_t$ and
$dE_t = \big[-\alpha E_t-\alpha\eta\Psi(Y_t)\big]dt + \widetilde\Sigma_t d\widetilde W_t$.

Apply It\^o's formula to the realized-state Bregman term $D_{\phi^*}(X_t,x^\star)$, $
dD_{\phi^*}(X_t,x^\star)
=
-\langle Y_t-y^\star,\Psi(\hat Y_t)\rangle dt 
+ \frac{1}{2}\operatorname{Tr}\big(  \nabla^2\phi^*(X_t)\Sigma_t \Sigma_t^\top\big)dt 
 + \langle Y_t-y^\star,\Sigma_t dW_t\rangle.$
Split the drift term as
$
-\langle Y_t-y^\star,\Psi(\hat Y_t)\rangle
=
-\langle Y_t-y^\star,\Psi(Y_t)\rangle 
-\langle Y_t-y^\star,\Psi(\hat Y_t)-\Psi(Y_t)\rangle.$
By Assumption~\ref{ass:local-vs},
$
-\langle Y_t-y^\star,\Psi(Y_t)\rangle
\le
-c_{\mathrm{vs}}D_{\phi^*}(X_t,x^\star).$
For the mixed term, Assumption~\ref{ass:local-reg} and \eqref{eq:y-d-bound} give
\begin{equation}
\label{eq:breg-mixed-bound-app-1}
\begin{aligned}
\big|\langle Y_t-y^\star,\Psi(\hat Y_t)-\Psi(Y_t)\rangle\big|
&\le \|Y_t-y^\star\|\,\|\Psi(\hat Y_t)-\Psi(Y_t)\| \\
\le L_\Psi\|Y_t-y^\star\|\,\|\hat Y_t-Y_t\| 
&\le L_\Psi M_\phi\|Y_t-y^\star\|\,\|E_t\| \\
\le C_{\mathrm{mix}}\sqrt{D_{\phi^*}(X_t,x^\star)} & \,\|E_t\|, 
\end{aligned}
\end{equation}
where $C_{\mathrm{mix}}=L_\Psi M_\phi\sqrt{C_\phi}$.
Applying Young's inequality with parameter $\mu>0$ yields
$
\big|\langle Y_t-y^\star,\Psi(\hat Y_t)-\Psi(Y_t)\rangle\big|
\le\;
\frac{\mu C_{\mathrm{mix}}^2}{2}D_{\phi^*}(X_t,x^\star) 
+\frac{1}{2\mu}\|E_t\|^2.$
Combining above inequalities and the trace bound
$
\frac{1}{2}\operatorname{Tr}\big(\nabla^2\phi^*(X_t)\Sigma_t\Sigma_t^\top\big)
\le \frac{M_\phi}{2}\sigma_X^2$
gives, with $c_D=c_{\mathrm{vs}}-\frac{\mu C_{\mathrm{mix}}^2}{2}$,
\begin{equation}
\label{eq:breg-final-drift-app}
\mathcal L D_{\phi^*}(X_t,x^\star)
\le
-c_D D_{\phi^*}(X_t,x^\star)
+
\frac{1}{2\mu}\|E_t\|^2
+
\frac{M_\phi}{2}\sigma_X^2.
\end{equation}

Now apply It\^o's formula to the corrective energy $\frac{\kappa}{2}\|E_t\|^2$:
\begin{equation}
\label{eq:e-energy-app}
\begin{aligned}
d\frac{\kappa}{2}\|E_t\|^2
&=
\kappa\langle E_t,dE_t\rangle
+ \frac{\kappa}{2}\operatorname{Tr}(\widetilde\Sigma_t\widetilde\Sigma_t^\top)dt \\
&=
\Big[-\kappa\alpha\|E_t\|^2-\kappa\alpha\eta\langle E_t,\Psi(Y_t)\rangle\Big]dt \\
&\quad + \frac{\kappa}{2}\operatorname{Tr}(\widetilde\Sigma_t\widetilde\Sigma_t^\top)dt
+ \kappa\langle E_t,\widetilde\Sigma_t d\widetilde W_t\rangle.
\end{aligned}
\end{equation}
Using $\|\Psi(Y_t)\|\le B_\Psi$ on the stopped neighborhood and Young's inequality,
$
-\kappa\alpha\eta\langle E_t,\Psi(Y_t)\rangle
\le
\kappa\alpha\eta B_\Psi\|E_t\|
\le
\frac{\kappa\alpha}{2}\|E_t\|^2 + \frac{\kappa\alpha\eta^2 B_\Psi^2}{2}$.
Hence
\begin{equation}
\label{eq:e-final-drift-app}
\mathcal L\frac{\kappa}{2}\|E_t\|^2
\le
-\frac{\kappa\alpha}{2}\|E_t\|^2
+
\frac{\kappa\alpha\eta^2 B_\Psi^2}{2}
+
\frac{\kappa}{2}\sigma_E^2.
\end{equation}

Finally, combine \eqref{eq:breg-final-drift-app} and \eqref{eq:e-final-drift-app}:
\begin{equation}
\label{eq:lyap-generator-app}
\begin{aligned}
\mathcal L\mathcal V(X_t,E_t)
&\le
-c_D D_{\phi^*}(X_t,x^\star)
-\left(\frac{\kappa\alpha}{2}-\frac{1}{2\mu}\right)\|E_t\|^2 \\
&\quad
+\frac{M_\phi}{2}\sigma_X^2
+\frac{\kappa}{2}\sigma_E^2
+\frac{\kappa\alpha\eta^2 B_\Psi^2}{2}.
\end{aligned}
\end{equation}
This is exactly \eqref{eq:local-drift-split} with $c_E=\frac{\kappa\alpha}{2}-\frac{1}{2\mu}$ and $b$ given by \cref{lem:local-drift}. Since
$\mathcal V(X_t,E_t)=D_{\phi^*}(X_t,x^\star)+\frac{\kappa}{2}\|E_t\|^2$,
we also obtain \eqref{eq:local-drift-main} with
$\lambda=\min\{c_D,2c_E/\kappa\}$. Under \eqref{eq:drift-feasibility}, the interval
$\left(\frac{1}{\kappa\alpha},\frac{2c_{\mathrm{vs}}}{C_{\mathrm{mix}}^2}\right)$
is nonempty, so $\mu$ can indeed be chosen to make both $c_D$ and $c_E$ positive.

\subsection{Proof for Theorem~\ref{thm:high-prob-terminal}}
\label{app:high-prob-proof}

From the semimartingale decomposition in Lemma~\ref{lem:local-drift}, on $[0,\tau_r]$,
$
d\mathcal V(Z_t)
\le \big(-\lambda\mathcal V(Z_t)+b\big)dt + dM_t,
$
where the local martingale part is
\begin{equation}
\label{eq:martingale-def}
dM_t
=
\langle Y_t-y^\star,\Sigma_t dW_t\rangle
+
\kappa\langle E_t,\widetilde\Sigma_t d\widetilde W_t\rangle.
\end{equation}
On $\{\tau_r>T\}$, multiplying by $e^{\lambda t}$ and integrating from $0$ to $T$ gives
$\mathcal V(Z_T)
\le e^{-\lambda T}\mathcal V(Z_0)
+\frac{b}{\lambda}(1-e^{-\lambda T})+N_T$,
where $N_T \triangleq  \int_0^{T\wedge\tau_r} e^{-\lambda(T-s)}dM_s$.

On the stopped region $\mathcal K_r$, the sublevel-set definition gives
$D_{\phi^*}(X_t,x^\star)\le r$ and $\|E_t\|^2 \le \frac{2r}{\kappa}$.
Using \eqref{eq:y-d-bound},
\begin{equation}
\label{eq:y-bound-app}
\|Y_t-y^\star\|^2
\le
C_\phi D_{\phi^*}(X_t,x^\star)
\le C_\phi r.
\end{equation}
Therefore the predictable quadratic variation of $N_T$ is bounded by
$
R_s\triangleq \|\Sigma_s\|_F^2\|Y_s-y^\star\|^2+\kappa^2\|\widetilde\Sigma_s\|_F^2\|E_s\|^2.$
Hence
\begin{equation}
\label{eq:qv-bound}
\relax[N]_T
\le
\int_0^{T\wedge\tau_r} e^{-2\lambda(T-s)} R_s\,ds
\le
\nu_r^2\frac{1-e^{-2\lambda T}}{2\lambda}.
\end{equation}
Also,
\[
\|\Sigma_s^\top(Y_s-y^\star)\|^2+\kappa^2\|\widetilde\Sigma_s^\top E_s\|^2 \le R_s.
\]
Set $\bar\nu_T^2\triangleq \nu_r^2(1-e^{-2\lambda T})/(2\lambda)$.
Since $N_T$ is a continuous martingale with $\relax[N]_T\le\bar\nu_T^2$, the exponential martingale inequality yields
$\mathbb P\!\left(N_T \ge \sqrt{2\bar\nu_T^2\log\frac{1}{\delta}}\right) \le \delta$.
Combining this inequality with \eqref{eq:exit-prob} and the union bound proves \eqref{eq:high-prob-main}.

\end{document}